\documentclass[11pt]{amsart}

\usepackage[ansinew]{inputenc}
\usepackage[a4paper,dvips]{geometry}
\usepackage{latexsym}
\usepackage{amsfonts}
\usepackage{amsmath,amssymb}
\usepackage{bbm}
\usepackage{color}
\usepackage{tikz}

\usepackage{graphicx}
\newtheorem{theorem}{Theorem}
\newtheorem{corollary}{Corollary}

\newtheorem{lemma}{Lemma}
\newtheorem{remark}{Remark}

\newcommand{\NN}{\mathbb{N}}

\newcommand{\ZZ}{\mathbb{Z}}
\newcommand{\RR}{\mathbb{R}}
\newcommand{\CC}{\mathbb{C}}

\newcommand{\ba}{\mathbf{a}}

\title{{\large On the symmetry of finite sums of exponentials}}

\author{Florian Pausinger}
\address{School of Mathematics \& Physics, Queen's University Belfast, BT7 1NN, Belfast, United Kingdom.}
\email{f.pausinger@qub.ac.uk}

\author{Dimitris Vartziotis}
\address{NIKI Ltd. Digital Engineering, Research Center, Ethnikis Antistasis 205, 45500 Katsikas, Ioannina, Greece.}
\email{dimitris.vartziotis@nikitec.gr}

\date{\today}

\begin{document}

\maketitle


\begin{abstract}
In this note we are interested in the rich geometry of the graph of a curve $\gamma_{a,b}: [0,1] \rightarrow \CC$  defined as
\begin{equation*}
\gamma_{a,b}(t) = \exp(2\pi i a t)  + \exp(2\pi i b t),
\end{equation*}
in which $a,b$ are two different positive integers. It turns out that the sum of only two exponentials gives already rise to intriguing graphs. We determine the symmetry group and the points of self intersection of any such graph using only elementary arguments and describe various interesting phenomena that arise in the study of graphs of sums of more than two exponentials. \\[6pt]
\end{abstract}

\section{Introduction}
\label{sec1}

Complex exponentials $\exp(2 \pi i k)$, $k \in \ZZ$, play a crucial role in different areas of pure and applied mathematics. They are the main building blocks of Fourier series in classical harmonic analysis and appear in all kinds of important exponential sums; see f.e. \cite{grafakos2,grafakos1,olson} and references therein. Many deep problems in number theory are intimately linked to certain types of exponential sums and a lot of effort is put into finding precise bounds for the growth of such sums; see f.e. \cite{iwaniec,korobov,shparlinski} and references therein. We do not aim to give a comprehensive account of the importance of complex exponentials and exponential sums in mathematics; we rather wish to convey a feeling for the immense power and structural richness that can be encoded with sums of exponentials. It turns out that already the most simple object, namely the sum of two exponentials, gives a powerful method to illustrate symmetry groups in a very easy fashion.  The aim of this note is to study the graphs of curves defined via the sum of two exponentials; i.e. we determine the symmetry groups of such graphs and find all points of self intersection. Recently, also the graphs of infinite sums of such exponentials have been studied; see \cite{bohnet}. In the following we set up our notation in a slightly more general frame. This allows to formulate interesting problems in our final section which are beyond the scope of this paper.

Let $\ba = (a_0, \ldots, a_m)$ denote a vector of positive integers $a_0, \ldots, a_m$ with $m\geq 1$ and let $\gamma_{\ba}: \RR \rightarrow \CC$ be the (closed) curve defined as
\begin{equation} \label{def}
\gamma_{\ba}(t) = \exp(2\pi i a_0 t) + \ldots + \exp(2\pi i a_m t) = \sum_{j=0}^m \exp(2\pi i a_j t).
\end{equation}
Note that $\gamma_{\ba}$ is one-periodic, i.e., $\gamma_{\ba}(t)= \gamma_{\ba}(t+1)$ for all $t \in \RR$, since $\exp(2 \pi i a t)$, $t \in [0,1]$, $a \in \ZZ$, is a circle in the complex plane.
Moreover, we do not require the integers in $\ba$ to be distinct. If an integer appears more than once in $\ba$ then $\gamma_{\ba}$ corresponds to a weighted sum of exponentials (with integer weights).
Interestingly, the graph of $\gamma_{\ba}$ can get quite chaotic for an arbitrary choice of parameters. However, it turns out that this graph can also be highly symmetric; see Figure \ref{fig:motivation}.

In Section \ref{sec2} we study the relation between the symmetry of the graph of $\gamma_{\ba}$ and the structure of the generating vector $\ba$. 
We recall the concept of a symmetry group and determine generating vectors of arbitrary, but finite, length $m$ such that the symmetry groups of the corresponding graphs are the dihedral groups $D_n$ for $n \in \NN$; see Figure \ref{fig:sym}. Our theorem gives a complete description of the symmetry of the graph of $\gamma_{a,b}$ for distinct integers $a,b$.
In Section \ref{sec3} we study the self intersections of curves $\gamma_{a,b}$. The most interesting observation is that the arguments $t,t' \in [0,1]$ such that $\gamma_{a,b}(t)=\gamma_{a,b}(t')$ are of a very particular rational from.

\begin{figure}[h!]
\begin{center}
\includegraphics[scale=0.5]{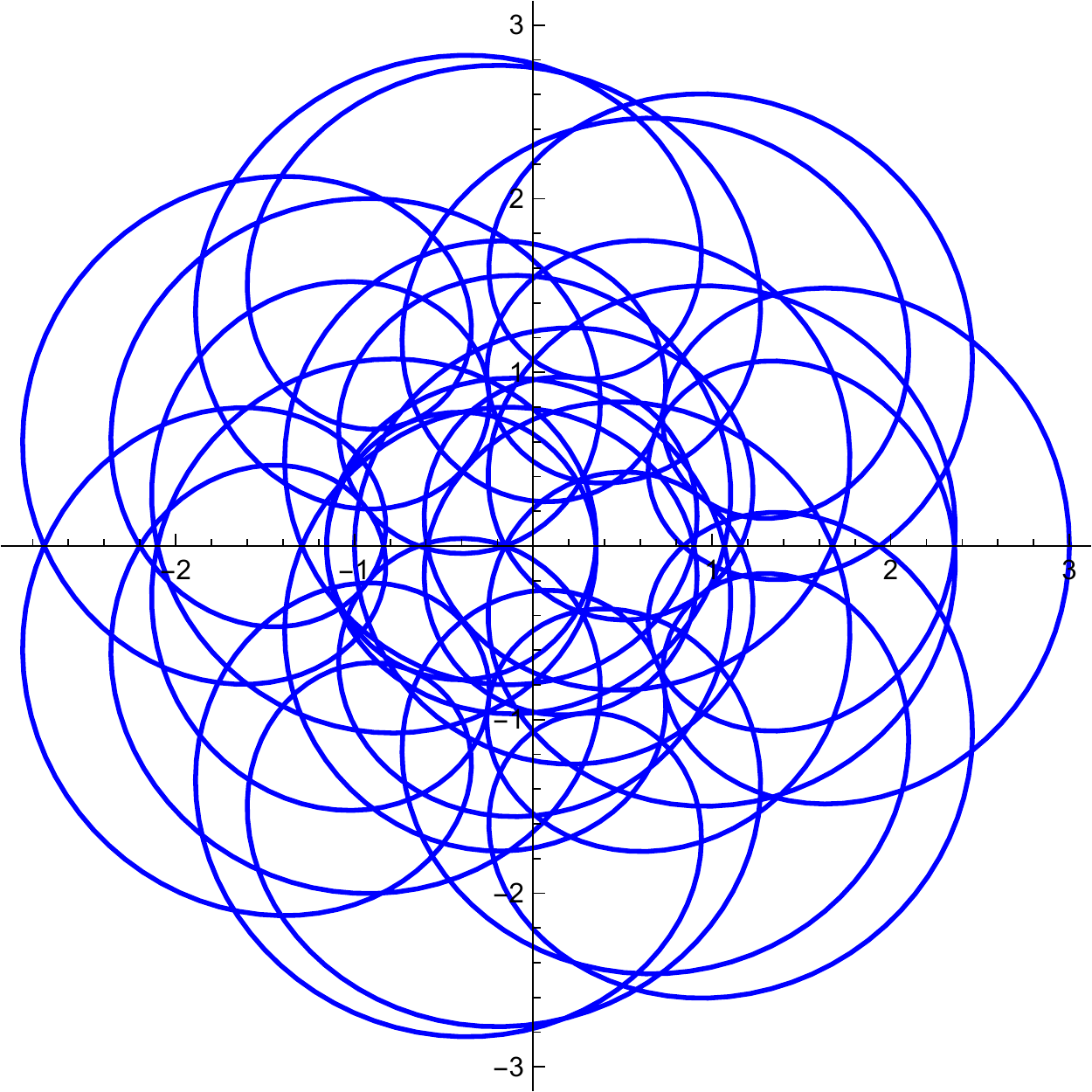}\
\includegraphics[scale=0.5]{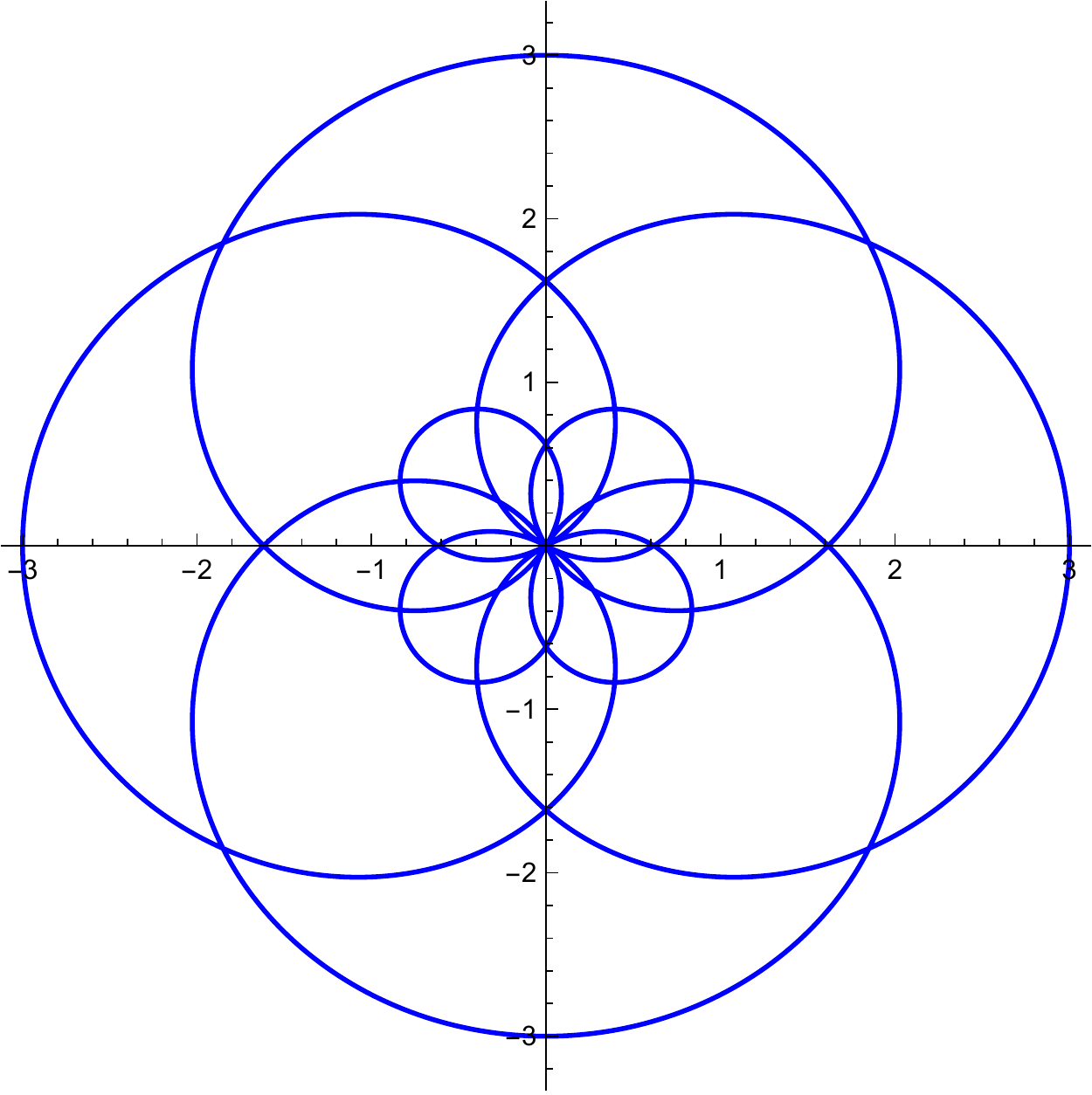}
\end{center}
\caption{The curves $\gamma_{2,7,31}$ and $\gamma_{1,5,9}$. } \label{fig:motivation}
\end{figure}

\section{Symmetry}
\label{sec2}

We start this section with a useful lemma, that allows to restrict attention to coprime tuples of numbers without loss of generality.

\begin{lemma}\label{lem:cancel}
Let $\ba=(a_0,a_1, \ldots, a_m)$ and let $q$ be a positive integer. Then $\gamma_{\ba}$ and $\gamma_{q\cdot \ba}$ have the same graph and
$$ \gamma_{\ba}(q t) = \gamma_{q\cdot \ba} (t).$$
\end{lemma}

\begin{proof}
We observe that 
\begin{equation*}
\gamma_{\ba}(q t) = \sum_{j=0}^m  \exp(2\pi i a_j q t) =\gamma_{q \cdot \ba}(t).
\end{equation*}
\end{proof}
Hence, the restriction of $\gamma_{q \cdot \ba}$ to $[0,1/q]$ gives the graph of $\gamma_{\ba}$ and so does the restriction to any other connected interval of length $1/q$.
In fact, the graph of $\gamma_{q \cdot \ba}$ covers the graph of $\gamma_{\ba}$ $q$ times for $t \in [0,1]$.

Next, we are interested in the symmetry of the graph of a curve $\gamma_{\ba}$. We use the concept of a \emph{symmetry group} to describe symmetric properties of a graph; at this point we would like to draw attention to the marvellous book of H. Weyl on symmetry \cite{weyl}. Let $G$ be a set and let $\cdot: G\times G \rightarrow G$ be a map that combines any two elements $a,b \in G$ to form another element $a\cdot b=c \in G$; i.e. the map is closed. The pair $(G,\cdot)$ is called a \emph{group} if 
\begin{enumerate}
\item the map is \emph{associative}, i.e. for all $a,b$ and $c \in G$, $(a\cdot b) \cdot c = a \cdot (b \cdot c)$;
\item there exists a unique \emph{identity element} $e$ such that $a\cdot e = e \cdot a = e$ holds for all $a \in G$;
\item for each $a \in G$ there exists an \emph{inverse element} $b \in G$, such that $a \cdot b = b \cdot a = e$.
\end{enumerate}
Now let $X \in \RR^2$ be the graph of a curve. The \emph{symmetry group} of $X$ consists of all transformations under which the object is \emph{invariant} (this means under which the object looks the same) with function composition as the group operation.
Important examples of symmetry groups in two dimensions are \emph{cyclic groups}, $C_n$, and \emph{dihedral groups}, $D_n$. Cyclic groups consist of all rotations about a fixed point by multiples of the angle $2\pi / n$; in the complex plane such a rotation is realized by multiplying every point of $X$ with $\exp(2\pi i j/n)$, for $1\leq j \leq n$. Dihedral groups contain $2n$ elements, namely the rotations in $C_n$ about a fixed point, together with reflections in $n$ axes that all pass through the fixed point of the rotations.
We illustrate these two different types of symmetries in Figure \ref{fig:symmetry}.
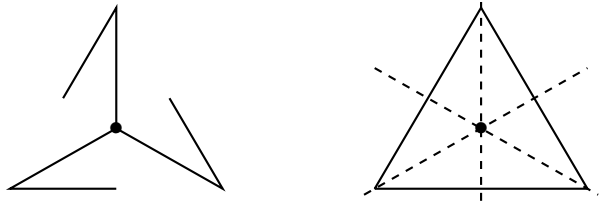
\begin{figure}[h!]
\begin{center}
\begin{tikzpicture}[scale=0.4]

\node at (3.5,2) {$\bullet$};
\draw[thick] (3.5,0) -- (0,0) -- (3.5,2)--(3.5,6) -- (1.75,3);
\draw[thick] (5.25,3) -- (7,0) -- (3.5,2);

\node at (15.5,2) {$\bullet$};
\draw[thick] (12,0) -- (19,0) -- (15.5,6)--(12,0);
\draw[thick, dashed] (11.65,-0.2) -- (19,4);
\draw[thick, dashed] (12,4) -- (19.35,-0.2);
\draw[thick, dashed] (15.5,-0.4) -- (15.5,6.5);

\end{tikzpicture}
\end{center}
\caption{\emph{Left:} Example of a shape with rotational symmetry. The symmetry group is $C_3$. \emph{Right:} The symmetry group of an equilateral triangle is $D_3$. The dashed lines should illustrate the axes of reflection.} \label{fig:symmetry}
\end{figure}

\begin{theorem} \label{thm:symmetry} Let $a, k, q, m \in \NN$ such that $a,k$ are coprime and set $q \cdot \ba=q \cdot (a,k+a,2k+a, \ldots, mk+a)$. Then the symmetry group of the graph of $\gamma_{q \cdot \ba}$ is the dihedral group $D_k$.
\end{theorem}

\begin{proof}
By Lemma \ref{lem:cancel} it suffices to study the case $\ba=(a,k+a,2k+a, \ldots, mk+a)$.
We prove the theorem in two steps. First, we show that $D_k$ is a subgroup of the symmetry group. In a second step we observe that $D_k$ contains already all invariant transformations of the image of $\gamma_{q \cdot \ba}$.

The group $D_k$ is a subgroup of the symmetry group if all rotations about a fixed point by multiples of the angle $2\pi/k$ leave the image invariant and if we find reflections in $k$ axes through the same fixed point.
We start with the rotations. We have to show that for all points $x = \gamma_{\ba}(t)$, $t\in [0,1]$, there exists a $t' \in [0,1]$ such that the point $\exp(2\pi i /k) x = \gamma_{\ba}(t')$. 
In the following we consider the special case $\ba=(1,k+1,2k+1, \ldots, mk+1)$.
Rotating a point $\gamma_{\ba}(t)$ by $\exp(2\pi i /k)$ leads to
\begin{align*}
\exp&(2\pi i /k) \cdot \gamma_{\ba}(t) =
\exp \left(2\pi i \left( t + \frac{1}{k} \right) \right) + \sum_{j=1}^{m} \exp \left(2\pi i (j k+1) \left( t + \frac{1}{(j k+1) k} \right) \right).
\end{align*}
To prove that the right hand side of the above equation lies again on $\gamma_{\ba}$ it suffices to find a $y$ such that $t'=t + \frac{y}{k}$ with
\begin{align*}
t + \frac{1}{k} & \equiv  t + \frac{y}{k} \pmod{1},\\
(k+1) \left( t + \frac{1}{k(k+1)} \right) & \equiv (k+1) \left( t + \frac{y}{k} \right) \pmod{1}, \\
&\vdots \\
(mk+1) \left( t + \frac{1}{(mk+1)k} \right) & \equiv  (mk+1) \left(  t + \frac{y}{k} \right) \pmod{1},\\
\end{align*}
Since $k$ and $m$ are integers, this amounts to finding a joint solution to the congruences
\begin{align*}
1 & \equiv y \pmod{k},\\
1 & \equiv (k+1) y \pmod{k}. \\
&\vdots \\
1 & \equiv (mk+1) y \pmod{k}.
\end{align*}
Obviously, $y=1$ is a valid joint solution to all these congruences. Hence, for every $t$ we can indeed find a $t'$ such that $\exp(2\pi i /k) \cdot \gamma_{\ba}(t) = \gamma_{\ba}(t')$.
Note that we can proceed in the same way in the general case. However, if $\ba=(a,k+a,2k+a, \ldots, mk+a)$, with $a>1$, we have to solve the congruence $1 \equiv ay \pmod{k}$. Since $a$ and $k$ are coprime we always find a solution (which will be different from $1$).

Turning to the reflections, it suffices to show that for every $x = \gamma_{\ba}(t)$, $t\in [0,1]$, there also exists a $t''$ such that $\bar{x} = \gamma_{\ba}(t'')$, in which $\bar{x}$ denotes the complex conjugate of $x$.
We recall that complex conjugation changes the sign of the exponent of an exponential. Thus $t$ maps to $-t$ and since our function is one-periodic, we can set $t''=1-t$ to see that complex conjugates of points in the image of $\gamma_{\ba}$ lie again on the curve.
This shows that $D_k$ is a subgroup of the full symmetry group and finishes the first step of our proof.

There are only two types of point symmetry groups in two dimensions; i.e. cyclic and dihedral groups. In particular, $D_k$ can only be a proper subgroup of a larger dihedral group of the form $D_{dk}$ for an integer $d>1$. If $D_k$ was a subgroup of a larger group then there have to be more invariant rotations. We will show that this is not the case with a simple geometric argument.

Recall that the absolute value $r=\sqrt{x^2+y^2}$ of a complex number $x+i y=r \exp(i \phi)$ denotes the Euclidean distance of the point $(x,y)$ to the origin. For two complex numbers $z_1=r_1 \exp(i \phi_1)$ and $z_2=r_2 \exp(i \phi_2)$ with $z_1+z_2=r \exp(i \phi)$ we get by the triangle inequality
$$  r \leq r_1 + r_2,  $$
with equality if and only if $\phi_1=\phi_2 \pmod{2\pi}$; see Figure \ref{fig:sum}. 
In particular, 
$$\exp(2 \pi i a t)+\exp(2 \pi i (a+k) t) = r \exp(i \phi),$$
with $0\leq r \leq 2$. In particular, we find that $r=2$ if and only if $2 \pi a t \equiv 2 \pi (a+k) t \pmod{2\pi}$ which is only satisfied if $t=y/k$ for an integer $y$. Thus, there are only $k$ points with maximal distance of $2$ to the origin on the graph of our curve; see Figure \ref{fig:sym}. Hence, there can not be more than $k$ invariant rotations and therefore $D_k$ is indeed the full symmetry group of the graph.
\end{proof}

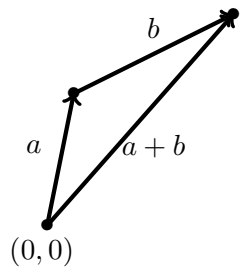
\begin{figure}[h!]
\begin{center}
\begin{tikzpicture}[scale=0.35]

\node at (0,0) {$\bullet$};
\node at (1,5) {$\bullet$};
\node at (7,8) {$\bullet$};
\draw[ultra thick, ->] (0,0) -- (1,5);
\draw[ultra thick, ->] (1,5)--(7,8);
\draw[ultra thick, ->] (0,0)--(7,8);

\node at (-0.5,3) {$a$};
\node at (4,7.5) {$b$};
\node at (4,3) {$a+b$};
\node at (-0.2,-1) {$(0,0)$};

\end{tikzpicture}
\end{center}
\caption{Two vectors $a$ and $b$ and their sum $a+b$. The length of $a+b$ is maximal if $a$ and $b$ point in the same direction.} \label{fig:sum}
\end{figure}

Together with Lemma \ref{lem:cancel} this theorem suffices to completely describe the symmetry of sums of two exponentials $\gamma_{a,b}$ for arbitrary integers $a$ and $b$. We can reduce the pair $(a,b)$ via Lemma \ref{lem:cancel} to a pair of coprime integers and apply the following corollary.

\begin{corollary} \label{cor:2d}
For coprime integers $a,b \in \NN$, with $b>a$ set $\ba = (a, b)$. If $a=1$, the symmetry group of the image of $\gamma_{\ba}$ is $D_{b-1}$. Otherwise the symmetry group of the image is $D_{b-a}$.
\end{corollary}

\begin{proof}
In the first case we have that $(a, b)=(1,b) = (1, (b-1)+1)$. In the second case we have that $(a, b)=(a, (b-a) + a )$.
The claims then follow from Theorem \ref{thm:symmetry}.
\end{proof}

\begin{figure}[h!]
\begin{center}
\includegraphics[scale=0.5]{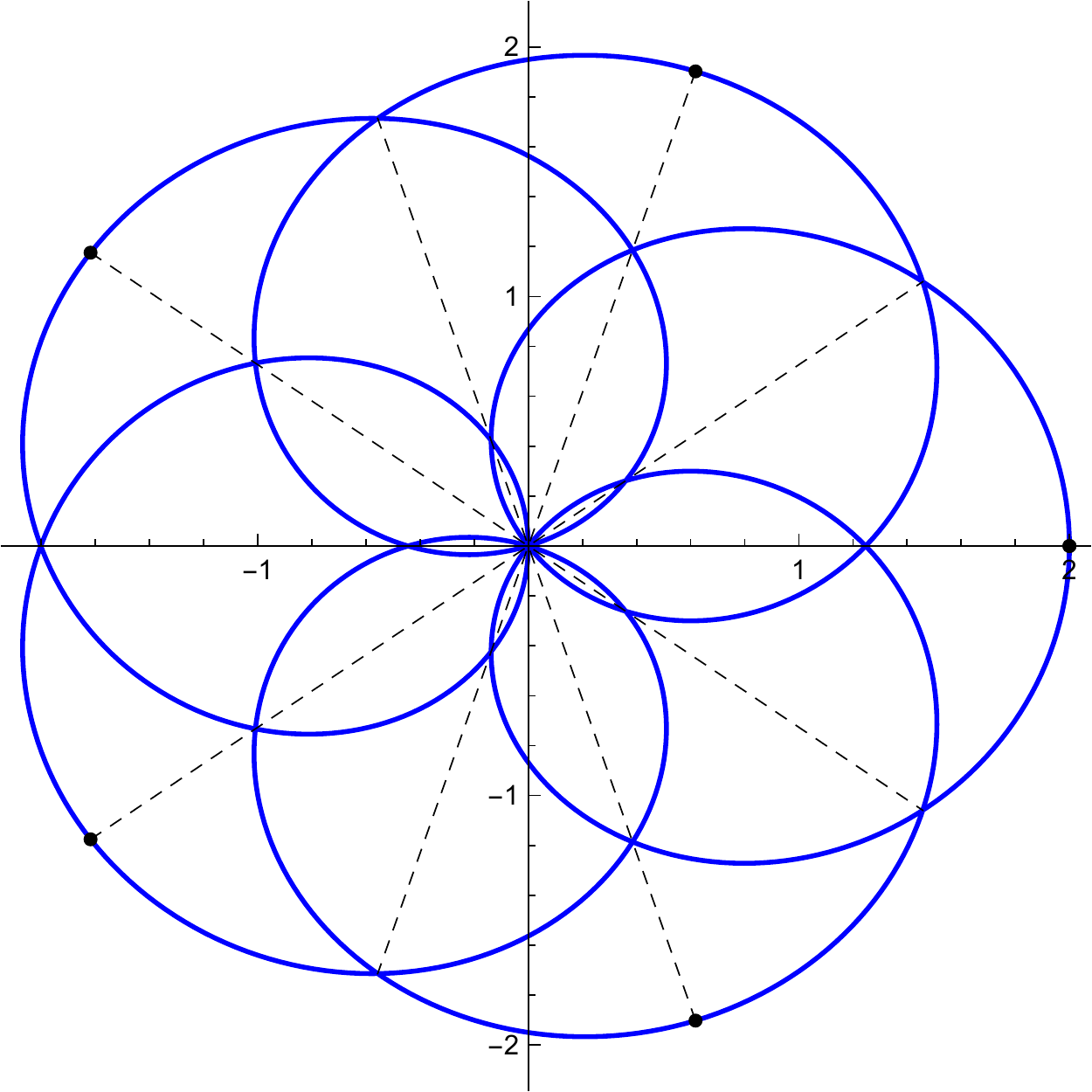}\
\includegraphics[scale=0.5]{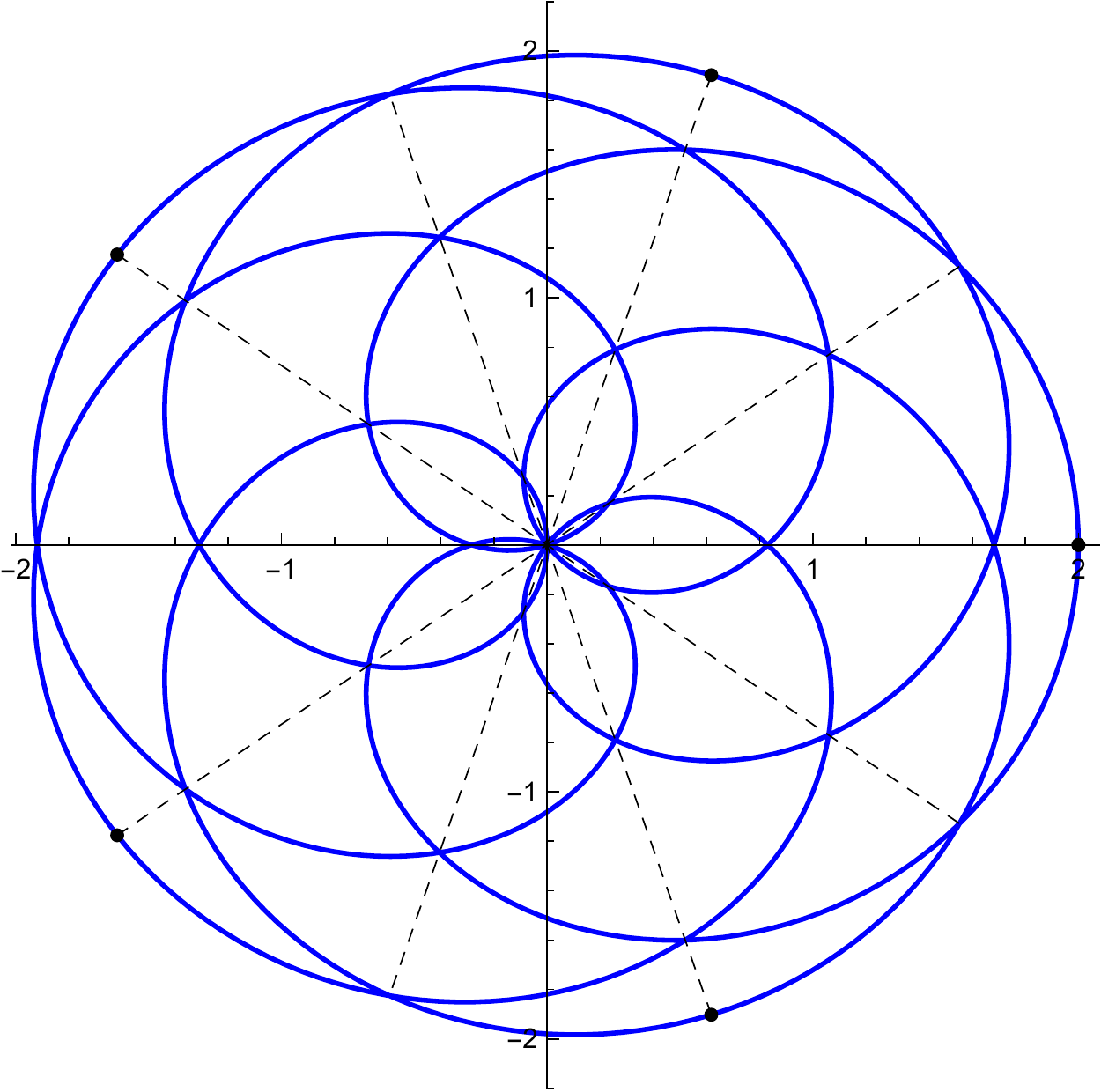}
\end{center}
\caption{The curves $\gamma_{1,6}$ and $\gamma_{3,8}$ and their reflection lines. The black dots are the $5$ points with maximal distance $2$ from the origin.} \label{fig:sym}
\end{figure}





\section{Self Intersections}
\label{sec3}

Our next goal is to understand the self intersections of $\gamma_{a,b}$ for coprime $a$ and $b$. We prove that such curves self intersect in a very structured way. Because of the symmetry of the graph, it suffices to study what happens in the \emph{wedge} spanned by the real line and by the line through the origin that is obtained by rotating the real line by an angle of $2\pi/(b-a)$. We prove the following theorem for the special case $\ba=(1,k+1)$, $k\geq 4$, and comment on the general case in Remark \ref{rem1}.

\begin{theorem} \label{thm:intersections}
Let $k\geq4$ be an integer and $\ba=(1,k+1)$. If $\gamma_{\ba}(t)\neq 0$ and if $\gamma_{\ba}(t)=\gamma_{\ba}(t')$ with $t, t' \in [0,1]$ and $t \neq t'$  then there exist integers $0\leq j,j' \leq k(k+2)$ such that $t=j/(k (k+2))$ and $t'=j'/(k (k+2))$.
\end{theorem}

We illustrate Theorem \ref{thm:intersections} in Figure \ref{fig:intersections}. The proof of the theorem requires a bit of preparation. We first state several observations in separate lemmas before we bring all these ideas together in the proof of Theorem  \ref{thm:intersections} at the end of this section.
In Lemma \ref{lem:iffpoints} we determine all values of $t$ such that the real or imaginary part of $\gamma_{a,b}(t)$ equals zero. 

\begin{figure}[h!]
\begin{center}
\includegraphics[scale=0.333]{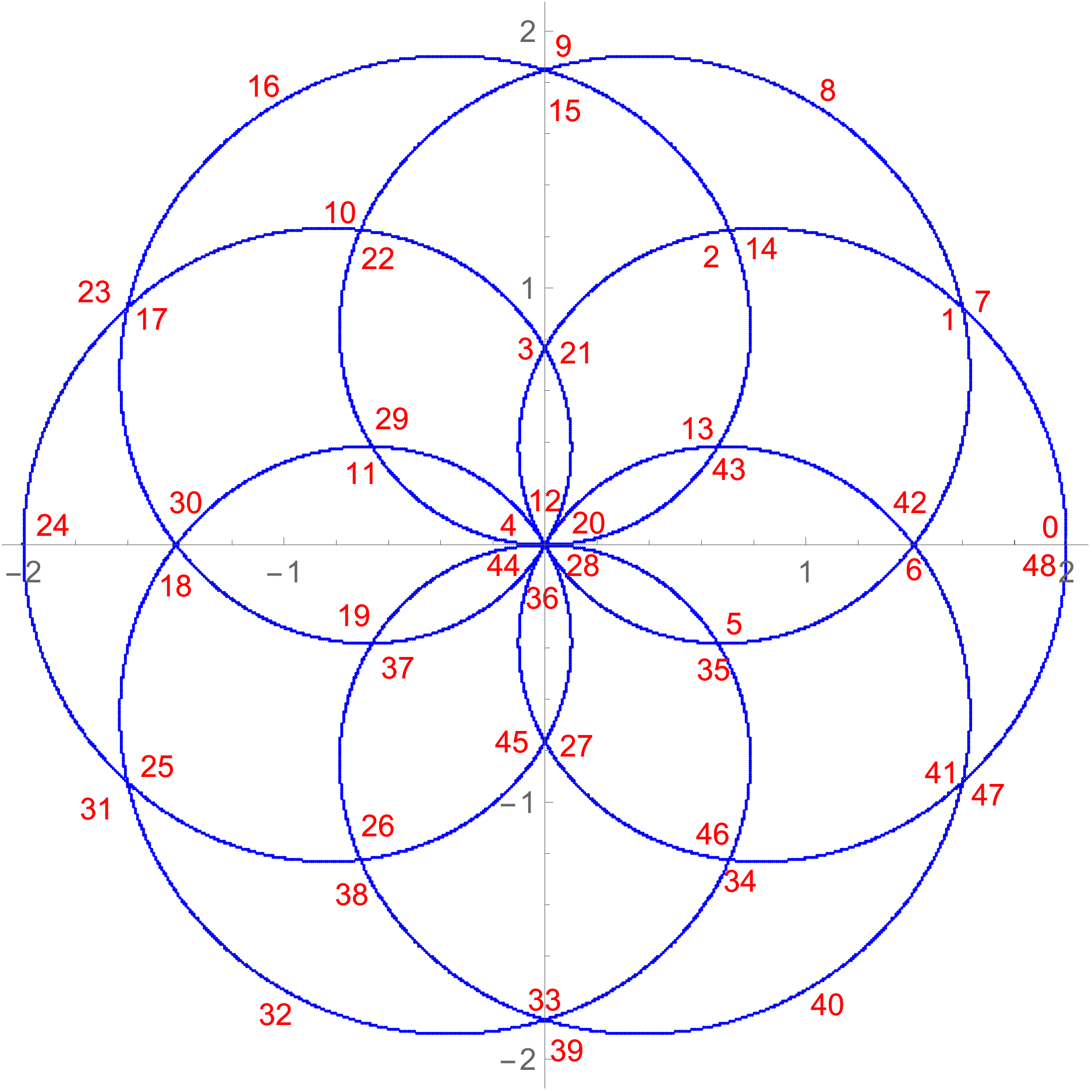}
\includegraphics[scale=0.5]{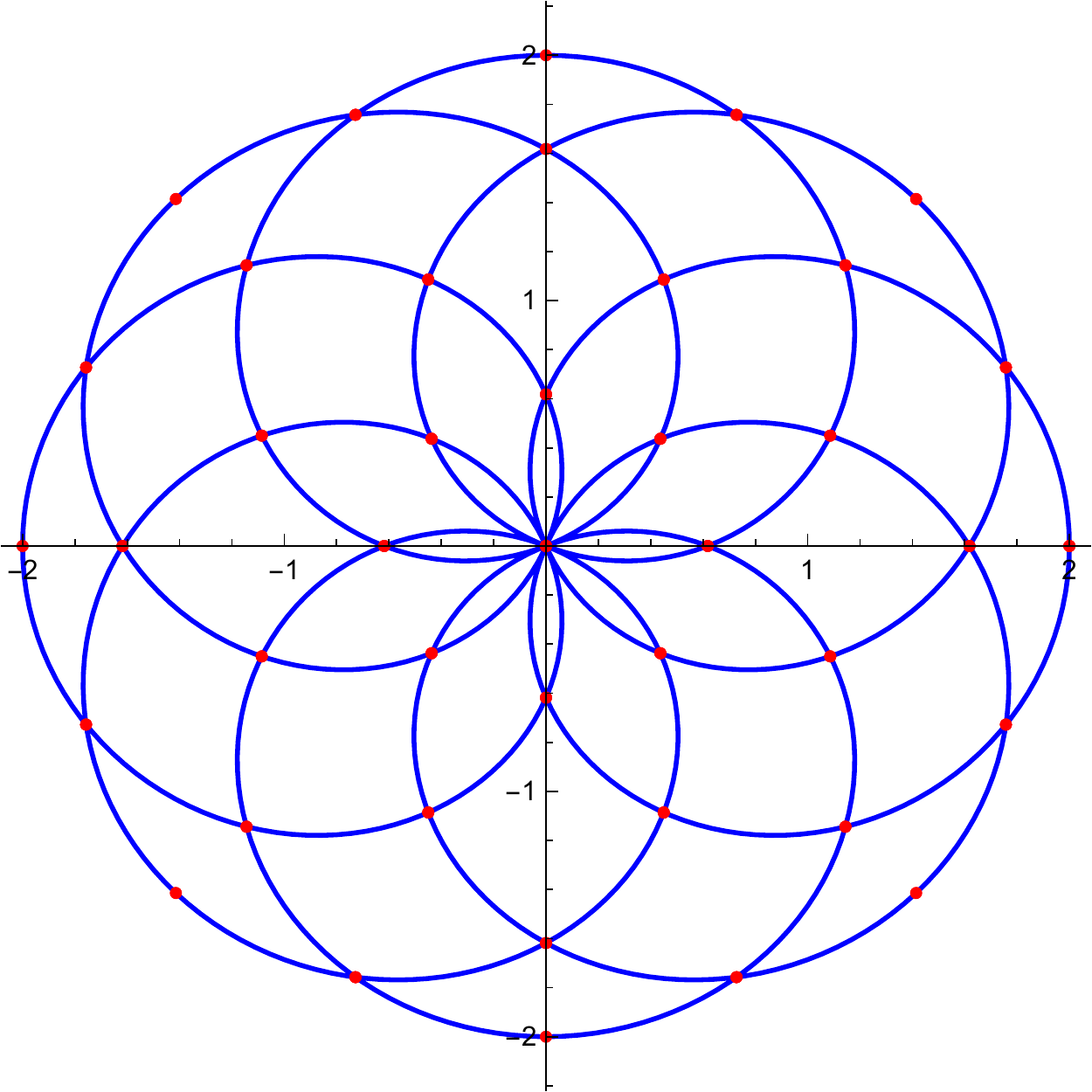}
\end{center}
\caption{Left: Plot of the curve $\gamma_{1,7}$. The integers are the numerators $j$ in $t=j/48$ of the points of self intersection. Right: Plot of the curve $\gamma_{1,9}$ in which the dots are at $t=j/80$.} \label{fig:intersections}
\end{figure}

\begin{lemma}\label{lem:iffpoints}
Let $a,b$ be coprime integers and $t \in [0,1]$. Set $\bar{t}_j=\frac{j}{2(a+b)}$, $j=0,\ldots, 2(a+b)$ and $\hat{t}_h= \frac{h}{2(b-a)}$ for $h=0,\ldots, 2(b-a)$. Then $\mathrm{Im}(\gamma_{a,b}( t ) ) = 0$ if and only if $t=\bar{t}_j$ for even $j$ or $t=\hat{t}_h$ for odd $h$. Moreover, $\mathrm{Re}(\gamma_{a,b}( t) ) = 0$ if and only if $t=\bar{t}_j$ for odd $j$ or $t=\hat{t}_h$ for odd $h$. In particular, $\gamma_{a,b}(t) =0$ if and only if $t=\hat{t}_h$ for odd $h$.
\end{lemma}

\begin{proof}
We first show that the real and imaginary parts vanish if we set $t=\bar{t}_j$ resp. $t=\hat{t}_h$ for $j$ and $h$ as stated. Observe that $\frac{a}{a+b} = 1 - \frac{b}{a+b}$. Now let $j=2q$ be even, then
\begin{align*}
\gamma_{a,b}(\bar{t}_j) &= \exp \left(2 \pi i a \frac{q}{a+b} \right) + \exp \left(2 \pi i b \frac{q}{a+b} \right) \\
&= \exp \left( 2 \pi i q \left(1- \frac{b}{a+b} \right) \right) + \exp \left(2 \pi i q \frac{b}{a+b} \right) \\
&= \exp \left( 2 \pi i q \right) \exp \left(- 2 \pi i q\frac{b}{a+b}  \right) + \exp \left( 2 \pi i q \frac{b}{a+b} \right) \\
&= 2 \cos \left(\frac{2 \pi q b}{a+b} \right).
\end{align*}
And similarly for $j=2q+1$ as well as for $\hat{t}_h$, which proves the first direction.

To prove the other direction, we first assume that $\gamma_{a,b}(t)=0$. This means
\begin{equation*}
\cos (2 \pi t a) + \cos (2 \pi t b)=0, \ \ \ \ \text{ and } \ \ \ \ \sin (2 \pi t a) + \sin (2 \pi t b) =0,
\end{equation*}
and can only be satisfied if $2\pi t a + \pi \equiv 2 \pi t b \pmod{2 \pi}$. Consequently, $ 1 \equiv 2 t (b-a) \pmod{2}$ and hence $t$ must be of the form $j/(2 (b-a))$ for an odd integer $j$.

Now assume that $\mathrm{Im}(\gamma_{a,b}( t ) ) = 0$ and $\mathrm{Re}(\gamma_{a,b}( t ) ) \neq 0$. Then we get that
\begin{equation*}
\sin (2 \pi t a) + \sin (2 \pi t b) =0.
\end{equation*}
Using $\sin(\varphi)=-\sin (-\varphi)$ we get,
$- 2 \pi t a \equiv 2 \pi t b \pmod{2 \pi}$ which implies $t (a+b) \equiv 0 \pmod{1}$ and therefore $t=j/(a+b)$ for an integer $j$.
Similarly, assume that $\mathrm{Re}(\gamma_{a,b}( t ) ) = 0$ and $\mathrm{Im}(\gamma_{a,b}( t ) ) \neq 0$. Then
\begin{equation*}
\cos (2 \pi t a) + \cos (2 \pi t b)=0.
\end{equation*}
Using $\cos(\pi/2+ \varphi) = - \cos(\pi/2 - \varphi)$, we get that $\pi \equiv 2 \pi t(a+b) \pmod{2\pi}$ which implies $1 \equiv 2 t(a+b) \pmod{2}$ and hence $t=\frac{j}{2(a+b)}$ for an odd integer $j$.
\end{proof}

In a next step we determine the points of self intersection of $\gamma_{a,b}$ that lie on the real axis. The symmetry of the graph of $\gamma_{a,b}$ implies then that there are also points of self intersection on every line that is a rotation of the real line by an angle $2\pi j/(b-a)$. These are the lines of points $r \cdot \exp(2\pi i j/(b-a) )$ with $r \in \RR$ and $1\leq j \leq b-a$.

\begin{lemma}\label{lem:samepoints}
Let $a, b$ be coprime integers and let $t=j/(a+b)$ and $t'= j'/(a+b)$, with $0< j ,j' < a+b$ and $j\neq j'$. Then we have $\gamma_{a,b}(t)=\gamma_{a,b}(t')$ if and only if $j+j'=a+b$.
\end{lemma}

\begin{proof}
First, if $j+j'=a+b$, then $t+t'=1$ and thus $\gamma_{a,b}(t)$ is the complex conjugate of $\gamma_{a,b}(t')$. Since the imaginary parts are 0 by Lemma \ref{lem:iffpoints}, we get $\gamma_{a,b}(t)=\gamma_{a,b}(t')$.

Second, assume $\gamma_{a,b}(t)=\gamma_{a,b}(t')$. By Lemma \ref{lem:iffpoints} we know that the imaginary parts are both 0 and we get
\begin{equation} \label{eq:equal}
2 \cos \left( \frac{2 \pi j b}{a+b} \right) = 2 \cos \left( \frac{2 \pi j' b }{a+b} \right).
\end{equation}
This equality can only hold if the arguments $x,y$ of the cosines either satisfy $x \equiv y \pmod{2 \pi}$ or $x+y \equiv 0 \pmod{2\pi}$.
First, we assume that the arguments are the same modulo $2\pi$. In this case we can rewrite \eqref{eq:equal} as
\begin{equation*}
\frac{(j - j') b}{a+b} \equiv 0 \pmod{1}.
\end{equation*}
However, $b$ and $(a+b)$ are coprime since $a$ and $b$ are coprime and $0<j,j'< a+b$, such that our assumptions only allow for the trivial solution $j=j'$. 
The second possibility to satisfy \eqref{eq:equal} leads to
\begin{equation} \label{cond1}
\frac{(j + j') b}{a+b} \equiv 0 \pmod{1}.
\end{equation}
By our assumptions, there is only one solution for \eqref{cond1} in this case, namely $j+j' = a+b$.
\end{proof}

So far, we know all points of self intersection that lie on the real line or on a rotation of it by an element of the symmetry group. However, there might be more. To determine them, we study the shape of the curve in intervals of the form $\mathcal{I}_j = [t_j, t_j+ 2/(b^2 - a^2)]$, with $t_j = j/(a+b)$, $j\in \NN$ and $0\leq j \leq a+b$; see Figure \ref{fig:restriction}.

\begin{lemma} \label{lem5}
Let $a,b$ be coprime with $b-a \geq 4$ and let $t_j=j/(a+b)$, with $j\in \NN$ and $0\leq j \leq a+b$. Then $\gamma_{a,b}: \mathcal{I}_j \rightarrow \CC$ is injective for $t \in \mathcal{I}_j=[t_j, t_j + 2/(b^2 - a^2)]$.
\end{lemma}

\begin{proof}
Setting $\gamma_{a,b}(t) =  (\mathrm{Re}(\gamma_{a,b}(t)), \mathrm{Im}(\gamma_{a,b}(t)) )$, we consider $\gamma_{a,b}$ as a function in $\RR^2$. 
In particular we have that
\begin{align*}
\mathrm{Re}(\gamma_{a,b}(t_j+x/(b^2-a^2))) &= 2 \cos \left( \pi j + \pi \frac{x}{b-a} \right) \cos \left( \pi (a-b) \frac{j(b-a) + x}{b^2 - a^2} \right)\\
\mathrm{Im}(\gamma_{a,b}(t_j+x/(b^2-a^2))) &= 2 \sin \left( \pi j + \pi \frac{x}{b-a} \right) \cos \left( \pi (a-b) \frac{j(b-a) + x}{b^2 - a^2} \right).
\end{align*}

To show that the restrictions of $\gamma_{a,b}(t)$ to intervals $[t_j, t_j+2/(b^2-a^2)]$ are injective we assume there are two values $x, x' \in [0,2]$ such that
\begin{equation} \label{tangens1} \gamma_{a,b}(t_j+x/(b^2-a^2)) = \gamma_{a,b}(t_j+x'/(b^2-a^2)). \end{equation}
Note that $-2\leq x-x' \leq 2$. From the assumption \eqref{tangens1} it follows that
\begin{equation} \label{tangens2}
\tan \left( \pi j + \pi \frac{x}{b-a} \right) = \tan \left( \pi j + \pi \frac{x'}{b-a} \right),
\end{equation}
since $\tan(\alpha)=\cos(\alpha)/\sin(\alpha)$. However, the tangens is a $\pi$-periodic function. This means that \eqref{tangens2} can only hold if the two arguments are the same modulo $\pi$:
\begin{equation*}
\pi j + \pi \frac{x}{b-a} \equiv  \pi j + \pi \frac{x'}{b-a} \pmod{\pi},
\end{equation*}
which can be reduced to
\begin{equation*}
\frac{x-x'}{b-a} \equiv 0 \pmod{1}.
\end{equation*}
Since $b-a \geq 4$ and $-2\leq x-x' \leq 2$ this can only hold for $x=x'$.
\end{proof}

Importantly, we can say even more about the curve in the intervals where it is injective, namely that the curves are also convex.

\begin{lemma} \label{lem6}
For coprime integers $a,b$ and $t_j$ as above, the curve $\gamma_{a,b}: \mathcal{I}_j \rightarrow \CC$ is convex for $t \in \mathcal{I}_j=[t_j, t_j + 2/(b^2 - a^2)]$.
\end{lemma}

\begin{proof}
A regular plane simple (injective) curve is convex if and only if its curvature is either always non-negative or always non-positive. We can calculate the curvature $\kappa_{a,b}(t)$ of our parameterized curve $(x(t),y(t))=(\mathrm{Re}(\gamma_{a,b}(t)) , \mathrm{Im}(\gamma_{a,b}(t)) )$ as
\begin{align*}
\kappa_{a,b}(t) = \frac{x'(t)y''(t) - y'(t)x''(t)}{(x'(t)^2 + y'(t)^2 )^{3/2}} = \frac{a^3 + b^3 + a b (a + b) \cos( 2\pi (a-b)t  )  }{ (a^2 + b^2 + 2 a b \cos( 2 \pi (a-b)t ))^{3/2} } > 0,
\end{align*}
since $-1\leq \cos(x)\leq 1$ for all $x$ and $a^3 + b^3 - a b (a + b)=(a-b)^2 (a+b)$ and $a^2 + b^2 - 2 a b = (a-b)^2$.
\end{proof}

\begin{proof}[Proof of Theorem \ref{thm:intersections}]
In the following we use the shortcut $\gamma:=\gamma_{(1,k+1)}$.
By Theorem \ref{thm:symmetry} the graph of $\gamma$ has the symmetry group $D_k$. By Lemma \ref{lem:iffpoints} we know that if $t_j= j/(k+2)$, $j\in \NN$, then $\mathrm{Im}(\gamma(t_j) ) =0$ and
$$\gamma(t_j)= 2 \cos \left( 2\pi \frac{j(k+1)}{k+2}\right).$$ 
Now we observe that $j (k+1) \equiv -j \pmod{k+2}$. In other words, for increasing $j$, with $0 \leq j \leq \left \lfloor (k+2)/2 \right \rfloor$, we get a decreasing sequence of equally spaced arguments for the cosine. 
By the symmetry of the cosine we get
\begin{align*} 
 2\cos(0) > 2 \cos \left ( 2 \pi \frac{-1}{k+2} \right) &= 2 \cos \left ( 2 \pi \frac{1}{k+2} \right) \\
> 2 \cos \left (2 \pi \frac{-2}{k+2} \right)&=2 \cos \left (2 \pi  \frac{2}{k+2} \right) > \ldots > 2 \cos \left (\pi \right)
\end{align*}
The above can be rewritten as:
\begin{equation*}
\gamma(0) > \gamma \left ( \frac{1}{k+2} \right) = \gamma \left ( \frac{-1}{k+2} \right) > 
\gamma \left ( \frac{2}{k+2} \right)=\gamma \left ( \frac{-2}{k+2} \right) > \ldots > \gamma \left (\left \lfloor \frac{k+2}{2} \right \rfloor \frac{1}{k+2} \right).
\end{equation*}
Rotating a complex number around the origin does not change the distance to the origin and we know by Theorem \ref{thm:symmetry} that
$$ \exp\left( 2\pi i /k \right) \gamma(t) = \gamma\left( t + \frac{1}{k} \right);$$
i.e., all points $\gamma \left ( \frac{\pm j}{k+2} \right)$ on the real line have the corresponding point $\gamma \left ( \frac{\pm j}{k+2} + \frac{1}{k} \right)$ on the line $r \cdot \exp\left( 2 \pi i / k \right)$. 
Importantly, we have that
\begin{equation} \label{eq1} t_j + \frac{1}{k}=\frac{j}{k+2} + \frac{1}{k} = \frac{(j+1)k+2}{k(k+2)} = t_{j+1}+ \frac{2}{k(k+2)}\end{equation}
and
\begin{equation} \label{eq2} t_{-j} + \frac{1}{k} = \frac{-j}{k+2} + \frac{1}{k} = \frac{(-j+1)k+2}{k(k+2)} = t_{-j+1} + \frac{2}{k(k+2)}. \end{equation}
Our first observation is now that if $\gamma(j/(k+2))=\gamma(-j/k+2)$, then 
$\gamma\left( \frac{(j+1)k+2}{k(k+2)} \right)=\gamma \left( \frac{(-j+1)k+2}{k(k+2)} \right).$
We recall from Lemma \ref{lem6} that $\gamma$ is injective and convex for $t \in \mathcal{I}_j=[t_j, t_j + 2/(k(k+2))]$.
Then the second observation is that $\gamma$ maps the right endpoints of all intervals $\mathcal{I}_j$ to the rotated line, while the left endpoints are mapped to the real line.
Since the $t_j$ are the only values $t$ such that $\gamma(t)$ lies on the real line and is nonzero, we conclude that all pieces of $\gamma$ that lie within the wedge spanned by the two lines can be obtained via restricting $t \in [0,1]$ to the intervals $\mathcal{I}_j$; see also Figure \ref{fig:restriction}.

\begin{figure}[h!]
\begin{center}
\includegraphics[scale=0.5]{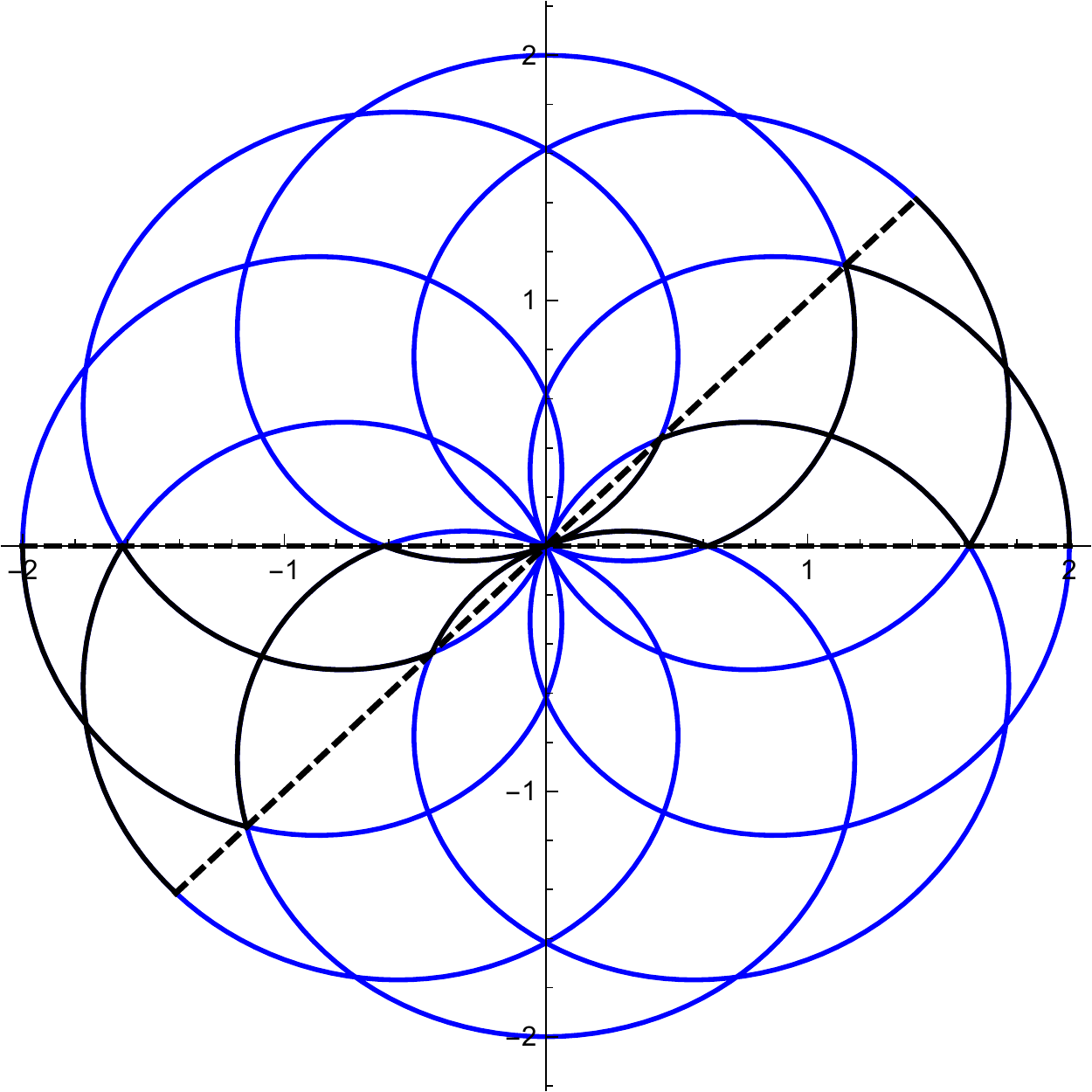}
\end{center}
\caption{Plot of the curve $\gamma_{1,9}$. The black parts show the restrictions of the curve to values of $t$ in intervals $\mathcal{I}_j$.} \label{fig:restriction}
\end{figure}

In particular we see, that if two left interval endpoints have the same image on the real line, the corresponding images of the right endpoints differ on the rotated line and vice versa.
By \eqref{eq1} and \eqref{eq2} it follows that rotating $\gamma(t_j)$ resp $\gamma(t_{-j})$ gives exactly the image of the right endpoints of $\mathcal{I}_{j+1}$ resp $\mathcal{I}_{-j+1}$; i.e. of the two intervals whose left endpoints are neighbors of $t_{\pm j}$ on the real line. 
Furthermore, the rotated images of $\gamma(t_{j-1})$ and $\gamma(t_{-j-1})$ are exactly the images of the right endpoints of $\mathcal{I}_{j}$ and $\mathcal{I}_{-j}$.
Hence, by convexity of the restrictions of the curve, we see that the curve must self intersect exactly $k$ times in the interior of the wedge and a direct calculation shows that the points of self intersection are exactly at the mid points of the intervals $\mathcal{I}_j$, i.e. at $ t_j + 1/(k(k+2))$.
\end{proof}

\begin{remark} \label{rem1}
This theorem also holds in the general case $\ba=(a,b)$; i.e. all arguments of points of self intersection are of the form $j/(b^2-a^2)$ for $0\leq j \leq b^2-a^2$.
However, the proof is more technical in the general case since the points $\gamma(t_j)$ are ordered in a more involved way. While the order and correspondence of the points in the above proof can easily be described via the congruence $j(k+1)\equiv -j \pmod{k+2}$, the general equation reads as $j b \equiv - aj \pmod{a+b}$.
For the sake of clarity of our exposition, we restricted the theorem to the case $(1,k+1)$. However, note that we proved all lemmas for the general case.
\end{remark}

\section{Outlook}
\label{sec4}

We conclude our exposition with a short outlook. One interesting direction for future research could be the study of weighted sums of exponentials. Starting again with the simplest possible case, we already obtain very interesting images when considering curves $\gamma_{a,b,b}$ as in Figure \ref{fig:weighted} (left). 
One obvious difference is that all the points on the curve have a certain minimal distance to the origin, while the graph seems to preserve the symmetry of $\gamma_{a,b}$.

A second intriguing question is to describe the points of self intersecting of general curves $\gamma_{\ba}$ as shown in Figure \ref{fig:weighted}(right). It seems that if $\ba$ has a structure as in Theorem \ref{thm:symmetry} such curves still intersect in a very well structured way; see Figure \ref{fig:motivation}(right). However, already in this case there seem to be some points of self intersection, which turn out to be very hard to describe in an explicit fashion.

\begin{figure}[h!]
\begin{center}
\includegraphics[scale=0.5]{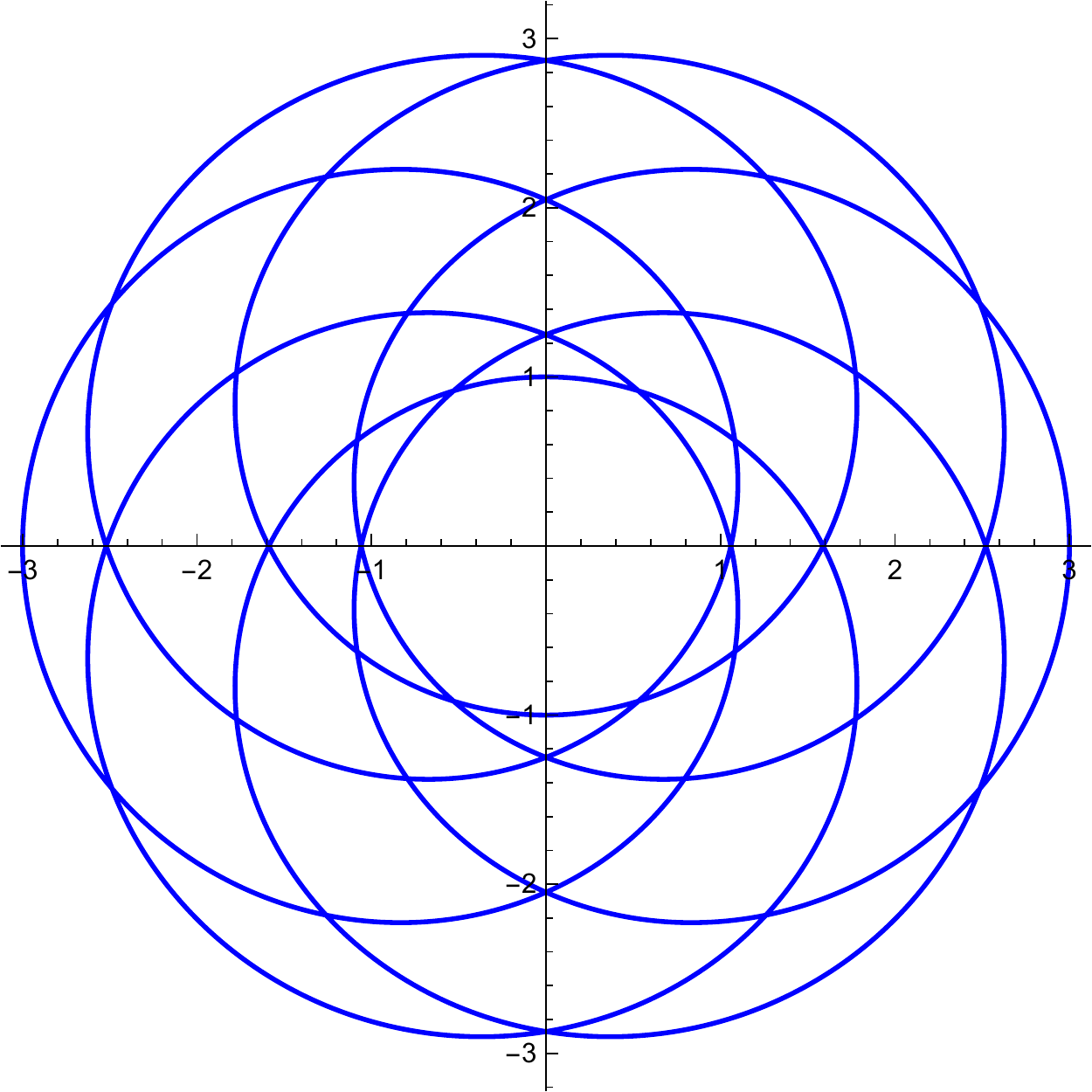}\ \
\includegraphics[scale=0.5]{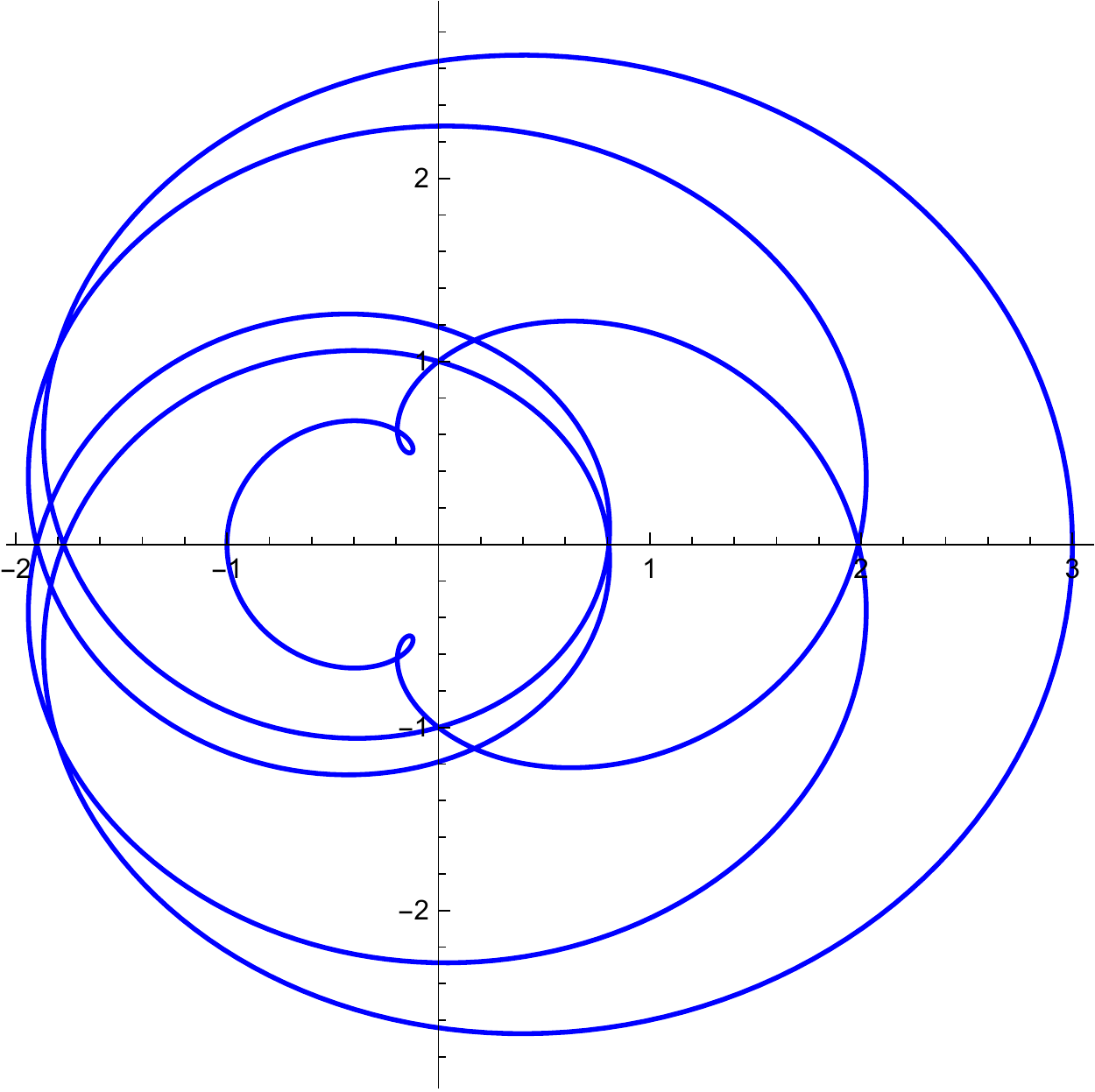}
\end{center}
\caption{Left: Plot of the curve $\gamma_{1,7,7}$. Right: Plot of the curve $\gamma_{3,6,7}$} \label{fig:weighted}
\end{figure}

\bibliographystyle{plain}
\bibliography{arxiv_Paus_Vartz_Color}

\begin{thebibliography}{1}

\bibitem{grafakos2}
Loukas {Grafakos}.
\newblock {\em {Classical Fourier analysis. 3rd ed.}}
\newblock New York, NY: Springer, 3rd ed. edition, 2014.

\bibitem{grafakos1}
Loukas {Grafakos}.
\newblock {\em {Modern Fourier analysis. 3rd ed.}}
\newblock New York, NY: Springer, 3rd ed. edition, 2014.

\bibitem{iwaniec}
Henryk {Iwaniec} and Emmanuel {Kowalski}.
\newblock {\em {Analytic number theory.}}
\newblock Providence, RI: American Mathematical Society (AMS), 2004.

\bibitem{korobov}
N.M. {Korobov}.
\newblock {\em {Exponential sums and their applications. Transl. from the
  Russian by Yu. N. Shakhov.}}
\newblock Dordrecht etc.: Kluwer Academic Publishers, 1992.

\bibitem{olson}
Tim {Olson}.
\newblock {\em {Applied Fourier analysis. From signal processing to medical
  imaging (to appear).}}
\newblock New York, NY: Birkh\"auser/Springer, 2017.

\bibitem{shparlinski}
Igor~E. {Shparlinski}.
\newblock {Open problems on exponential and character sums.}
\newblock In {\em {Number theory. Dreaming in dreams. Proceedings of the 5th
  China-Japan seminar, Higashi-Osaka, Japan, August 27--31, 2008}}, pages
  222--242. Hackensack, NJ: World Scientific, 2010.

\bibitem{bohnet}
Dimitris Vartziotis and Doris Bohnet.
\newblock Fractal curves from prime trigonometric series.
\newblock {\em Fractal and Fractional}, 2(1), 2018.

\bibitem{weyl}
Hermann {Weyl}.
\newblock {\em {Symmetry. Reprint of the 1952 edition.}}
\newblock Princeton, NJ: Princeton University Press, reprint of the 1952
  edition edition, 2016.

\end{thebibliography}


\end{document}